\documentclass[11pt]{amsart}
\usepackage{amsmath}
\usepackage{amssymb}
\usepackage{amsthm}
\usepackage{verbatim}
\newtheorem{thm}{Theorem}[section]
\newtheorem{cor}[thm]{Corollary}

\newtheorem{lem}[thm]{Lemma}

\theoremstyle{definition}

\theoremstyle{remark}
\newtheorem{rem}[thm]{Remark}
\newcommand{\ls}{\lesssim}

\def \T{\mathrm{T}}
\def \P{\mathrm{P}}
\def \M{\mathrm{M}}
\def \S{\mathrm{S}}
\def \I{\mathrm{I}}
\def \bi{\dot{B}^{-1}_{2,\infty}}
\def \b{\dot{B}^{1}_{2,1}}
\def \Z{\mathbb{Z}}

\def \D{\Omega}

\def\R{\mathbb{R}}
\def\2L{\Lambda_{\tilde{\gamma}}}
\def\1L{\Lambda_{\gamma}}

\def \d{\partial_z}
\def \dc{\partial_{\overline z}}

\def \la{\lambda}
\renewcommand{\le}{\leqslant}
\renewcommand{\ge}{\geqslant}
\renewcommand{\leq}{\leqslant}
\renewcommand{\geq}{\geqslant}

\begin{document}

\title[Unbounded potential recovery in the plane]{Unbounded potential recovery in the plane}

\thanks{Supported by the ERC grant 307179 and the MINECO grants MTM2011-28198, MTM2013-41780 and SEV-2011-0087 (Spain)}
\author{Kari Astala}
\address{Department of mathematics and statistics,
Po.Box 68, 00014, University of Helsinki, Finland} \email{kari.astala@helsinki.fi}
\author{Daniel Faraco}
\address{Departamento de Matem\'aticas - Universidad Aut\'onoma de Madrid
 and Instituto de Ciencias Matem\'aticas
CSIC-UAM-UC3M-UCM, 28049 Madrid, Spain} \email{daniel.faraco@uam.es}
\author{Keith M. Rogers}
\address{Instituto de Ciencias Matem\'aticas CSIC-UAM-UC3M-UCM, 28049 Madrid, Spain} \email{keith.rogers@icmat.es}
\dedicatory{Dedicated to the memory of Tuulikki}

\maketitle

\begin{abstract} We reconstruct compactly supported potentials with only half 
a derivative in $L^2$ from the scattering amplitude at a fixed energy. For this we draw a connection between the recently introduced method of Bukhgeim, which uniquely determined the potential from the Dirichlet-to-Neumann map,  and a question of Carleson regarding the convergence to
initial data of solutions to time-dependent Schr\"odinger equations.  We also provide examples of compactly supported potentials, with $s$ derivatives in $L^2$ for any $s<1/2$, which cannot 
be recovered  by these means. Thus the recovery method has a different threshold in terms of regularity than the corresponding uniqueness result.
\end{abstract}

\vspace{1em}

\section{Introduction}

We  consider the  Schr\"odinger equation $\Delta u=
V u$ on a bounded domain~$\D$ in the plane. For each solution $u$,  we are given the value of both $u$ and $\nabla u \cdot n$ on the boundary $\partial\D$,
where $n$ is the exterior unit normal on $\partial\D$. The goal is
then to recover the potential $V$ from this information.

We suppose throughout that $V\in L^2$ is supported on $\Omega$ and that $0$ is not a
Dirichlet eigenvalue for the Hamiltonian $-\Delta +V$.  Then for each
 $f\in H^{1/2}(\partial \Omega)$, there is a unique solution $u\in H^1(\Omega)$ to the Dirichlet problem
  \begin{equation}\label{dp}
\begin{cases}
\Delta u=Vu\\
u\big|_{\partial \Omega}=f,
\end{cases}
\end{equation}
and the Dirichlet-to-Neumann ({\small DN}) 
map $\Lambda_V$ can be formally defined by $$\Lambda_V\,:\,f\mapsto
\nabla u \cdot n|_{\partial\Omega}.$$
Then a restatement of our goal is to
recover  $V$ from knowledge of~$\Lambda_V$.

We come to this problem via a question of Calder\'on regarding impedance tomography~\cite{Calderon80}, where $f$ is the electric potential and $\nabla u \cdot n$ is the boundary current, however the {\small DN} map $\Lambda_{V-\kappa^2}$ and the scattering amplitude at energy~$\kappa^2$ are uniquely determined by each other, and indeed the {\small DN} map can be recovered from
the scattering  amplitude (see the appendix for explicit formulae). Thus we are also addressing the question of whether it is possible to recover a potential from the scattering data at a fixed positive energy.

In higher dimensions, Sylvester and Uhlmann  proved that smooth potentials  are uniquely determined by the {\small DN} map \cite{SylvesterUhlman87} (see \cite{NSU88,Novikov88,Ch} for nonsmooth potentials and \cite{BrownTorres03,PPU03,HabermanTatarupp} for the conductivity problem). The uniqueness result was extended to a reconstruction procedure by Nachman \cite{Nachman88,Nachman91}. The planar case is quite different mathematically as it is not overdetermined. Here the first uniqueness and reconstruction algorithm was proved by Nachman \cite{Nachman96}   via $\overline{\partial}$-methods for potentials of conductivity type (see also~\cite{BrownUhlman97} for uniqueness with less regularity). Sun  and Uhlmann~\cite{SU91,SU93} proved uniqueness for potentials satisfying nearness conditions to each other. Isakov and Nachman~\cite{IN} then reconstructed the real valued $L^{p}$-potentials,  $p>1$, in the case that their eigenvalues are strictly positive.  The $\overline{\partial}$-method in combination with the theory of quasiconformal maps gave the uniqueness result for the conductivity equation with measurable coefficients
\cite{AstalaPaivarinta06}.  The problem for the general Schr\"odinger equation was  solved only  in 2008 by Bukhgeim \cite{B} for $C^1$-potentials.
Bukhgeim's result has since been improved and extended to treat related inverse  problems (see for example \cite{Blastenpp, GST, GT0, GT, NS, IUY10, IY}).

The aim of this article is to emphasise a surprising connection between
the  pioneering work of Bukhgeim \cite{B} and Carleson's question
\cite{C} regarding the convergence to initial data of solutions to time-dependent
Schr\"odinger equations.  Elaborating
on this new point of view we obtain a reconstruction theorem for
general planar potentials with only half a derivative in $L^2$, which is
sharp with respect to the regularity.
The precise statements
are given in the forthcoming Corollary~\ref{dim} and Theorem~\ref{sharpness}.

To describe the results in more detail, we recall that the starting point in~\cite{B} was  to consider  solutions to $\Delta u=V u$
 of the form
$u = e^{i{\psi}}\big(1+w\big)$, where from now on
$$\psi(z)\equiv \psi_{k,x}(z)=\tfrac{k}{8}(z-x)^2,\qquad z\in \mathbb{C},\ \ x\in\Omega.$$
Solutions of this type have a long history (see for example \cite{F, SylvesterUhlman87, KSU, D}), and in this form they were considered first by Bukhgeim.
We will recover the potential by measuring a countable number of times  on the boundary, so we take $k\in \mathbb{N}$.
We will require the homogeneous Sobolev spaces with norm given by $\|f\|_{\dot{H}^s}=\|(-\Delta)^{s/2}f\|_{L^2},$ where
 $(-\Delta)^{s/2}$ is defined via the Fourier transform as usual.
In Section~\ref{rem}, we prove that if the potential $V$ is
contained in $\dot{H}^s$ with
$0<s<1$, and $k$ is sufficiently large, then we can take $w\equiv w_{k,x}\in
\dot{H}^s$ with a bound for the norms which is decreasing to zero in~$k$.
We write $u_{k,x}=e^{i{\psi}}\big(1+w\big)$ for these
$w\in \dot{H}^s$.

The definition of the {\small DN} map yields  the basic integral formula in
inverse problems; Alessandrini's identity. Indeed,
 if
 $u, v\in H^1(\D)$ satisfy $\Delta u=Vu$ and $\Delta v=0$, then the formula states that
\begin{equation*}\label{Alessandrini}
\Big\langle (\Lambda_V-\Lambda_0)[u],v
\Big\rangle:=\int_{\partial\Omega} (\Lambda_V-\Lambda_0)[u]\, v=
\int_\D  V u \,v.
\end{equation*}
Taking $u=u_{k,x}$, which is also in $H^1(\Omega)$,  and $v=e^{i\overline{{\psi}}}$ this yields
\begin{equation}\label{ale} \Big\langle
(\Lambda_V-\Lambda_0)[u_{k,x}],e^{i{\overline{\psi}}}
\,\Big\rangle=\int_{\D} e^{i({\psi}+\overline{{\psi}})}V(1+w)\,,
\end{equation}
and so the integral over $\Omega$ can be obtained from information on the boundary.

The bulk of the article is concerned with recovering the potential
from the integral on the right-hand side of \eqref{ale}. However, in
order to calculate the value of the integral, without knowing the
value of the potential $V$ inside~$\D$, we need to calculate the
value of the left-hand side of \eqref{ale}. That is to say, we must
determine the values of $u_{k,x}$ on the boundary  from the {\small DN} map.
In the case of linear phase, this was achieved by Nachman~\cite{Nachman96}  for $L^p$-potentials $V$, with $p>1$, and Lipschitz boundary. For $C^1$-potentials, with $C^2$-boundary, the result was extended by Novikov and Santacesaria  to quadratic phases~\cite{NS}. Here we show that for quadratic phases almost no regularity is needed. We consider potentials in the inhomogeneous $L^2$-Sobolev space $H^s$, defined as before with $(-\Delta)^{s/2}$ replaced by $(\I-\Delta)^{s/2}$. Our starting point is similar to \cite{Nachman96} but we give a shorter argument,  avoiding single layer potentials.

\begin{thm}\label{far} Let $V \in H^{s}$ with $s>0$ and suppose that $\Omega$ is Lipschitz. Then we can identify compact operators $\Gamma_{\!k,x}: H^{1/2}(\partial \Omega)\to H^{1/2}(\partial \Omega)$, depending only on $k, x$ and $\Lambda_V-\Lambda_0$, such that
\[u_{k,x}{|_{\partial \D}} =(\mathrm{I}-\Gamma_{\!k,x})^{-1}\big[e^{i{\psi}}{|_{\partial \D}}\big]. \]
\end{thm}
For $C^1$-potentials, Bukhgeim \cite{B} proved that the right-hand side
of \eqref{ale}, multiplied by $(4\pi)^{-1}k$,  converges to $V(x)$ for all $x\in \D$, when $k$ tends to infinity. In
Section~\ref{recovery}, we obtain this convergence for
potentials in $H^s$ with $s>1$. For discontinuous potentials we are no longer able to recover at each point. Instead we bound the fractal dimension of the sets where the recovery fails. As Sobolev spaces are only defined modulo sets of zero Lebesgue measure, we consider first the potential spaces~$L^{s,2}=(-\Delta)^{-s/2}L^2(\R^2)$, and bound
 the Hausdorff dimension of the points where the recovery fails.

\begin{thm}\label{dim2}
 Let $V \in  L^{s,2}$ with $1/2\le s< 1$. Then
\[ \dim_H\Big\{x\in\D\, :\, \tfrac{k}{4\pi} \Big\langle (\Lambda_V-\Lambda_0)[u_{k,x}], e^{i{\overline{\psi}}}\, \Big\rangle\not\to V(x)\ \ \text{as}\ \ k\to\infty\Big\}\le 2-s. \]\end{thm}

As the members of $H^s$ coincide almost everywhere
with members of~$L^{s,2}$, we see that rough and unbounded
potentials can be recovered almost
everywhere
 from information on the boundary.
Note that these results are stable in the sense that $k\in\mathbb{N}$ can be replaced by any sequence $\{n_k\}_{k\in\mathbb{N}}$ such that $n_k$ tends to infinity as $k$ tends to infinity.

\begin{cor}\label{dim} Let $V \in  {H}^{1/2}$. Then
\[ \lim_{k\to \infty} \tfrac{k}{4\pi} \Big\langle (\Lambda_V-\Lambda_0)[u_{k,x}], e^{i{\overline{\psi}}}\, \Big\rangle= V(x),\quad \emph{a.e.}\ x\in \D. \]
\end{cor}

  In Section~\ref{sharp}, we will prove that this is sharp in the
 sense of the following theorem. Note that even though there is divergence on a set of full Hausdorff dimension when $s<1/2$, the dimension of the divergence set is bounded above by $3/2$ when $s\ge 1/2$.

\begin{thm}\label{sharpness} Let $s<1/2$. Then there exists a potential $V\in  {H}^s$,  supported in $\Omega$, for which
\[ \Big|\Big\{ x \in\D\,:\, \tfrac{k}{4\pi} \Big\langle (\Lambda_{V}-\Lambda_0)[u_{k,x}], e^{i{\overline{\psi}}}\, \Big\rangle\not\to V(x)\ \ \text{as}\ \ k\to\infty\Big\}\Big| \neq 0.
\]
\end{thm}

Bl{\aa}sten \cite{Blastenpp} proved that potentials in $H^s$ with $s>0$ are uniquely determined by the {\small DN} map (see also \cite{IY} for uniqueness with $L^p$-potentials, $p>2$).  It is a curious phenomenon that, within the Bukhgeim approach, uniqueness and reconstruction have different smoothness barriers.

The {\small DN} map $\Lambda_{V-\kappa^2}$ can be recovered from the scattering amplitude at a fixed energy $\kappa^2>0$ (see the appendix),
  from which we are able to recover the potential $V-\kappa^2\chi_\Omega$ rather than $V$. We are free to choose the domain~$\Omega$. Taking $\Omega$ to be a square, we obtain the following recovery formula. Here $U_{k,x}$ are Bukhgeim solutions which solve $\Delta u=
(V -\kappa^2)u$ in $\Omega$.

\begin{thm}\label{names}
Let $V \in  {H}^{1/2}$  be supported in a square $\Omega$. Then
\[ \lim_{k\to \infty} \tfrac{k}{4\pi} \Big\langle (\Lambda_{V-\kappa^2}-\Lambda_0)[U_{k,x}], e^{i{\overline{\psi}}}\, \Big\rangle+\kappa^2= V(x),\quad \emph{a.e.}\ x\in \D. \]
\end{thm}

Interpreting the problem acoustically, it is unsurprising that we are unable to recover potentials in $H^s$ with $s<1/2$. Taking $$V(x)=\kappa^2(1-c^{-2}(x)),$$ where $c(x)$ denotes the speed of sound at~$x$, the scattered solutions $u$ also satisfy $c^2\Delta u + \kappa^2 u=0$. Now  there are potentials in $H^s$, with $s<1/2$, which are singular on closed curves (see for example~\cite{Z}). Thus the speed of sound is zero on the curve and so  a continuous solution~$u$ would be zero. That is to say, the continuous incident waves cannot pass through the curve and we should not expect to be able to detect modifications of the interior of the potential which is cloaked in some sense (see \cite{GKLU} for more sophisticated types of cloaking). From this point of view, the uniqueness results \cite{Blastenpp, IY} reflect the tunneling phenomenon in quantum mechanics.

\section{The Bukhgeim solutions}\label{rem}

Writing $\partial_z =\frac{1}{2}(\partial_x-i\partial_y)$ and $\partial_{\overline{z}} =\frac{1}{2}(\partial_x+i\partial_y)$,
we consider the complex analytic interpretation of the Schr\"odinger equation $4\partial_z\partial_{\overline{z}} u=V u$. When looking for solutions of the form $u=e^{i{\psi}}\big(1+w\big)$, the
equation is equivalent
to the system
$$
2\partial_{\overline{z}}w=e^{-i({\psi}+\overline{{\psi}})}v,\qquad 2\partial_z v=e^{i({\psi}+\overline{{\psi}})}V(1+w),
$$
which is solved in $\D$ whenever
$$
w=\tfrac{1}{4}\partial_{\overline{z}}^{-1}\Big[
e^{-i({\psi}+\overline{{\psi}})}\chi_{Q}\,\partial_{z}^{-1}\big[e^{i({\psi}+\overline{{\psi}})}\,V(1+w)\big]\Big].
$$
Here, we take $Q $ to be a fixed, auxiliary, axis-parallel square which properly contains $\D$.
Thus, defining the operator $\S^k_V\equiv \S^{k,x}_V$ by
\begin{equation*}\label{operator}
\S^k_V [F]=\tfrac{1}{4}\partial_{\overline{z}}^{-1}\Big[
e^{-i({\psi}+\overline{{\psi}})}\chi_{Q}\,\partial_{z}^{-1}\big[e^{i({\psi}+\overline{{\psi}})}\chi_Q\,VF\big]\Big],
\end{equation*}
we see that as soon as $\|\S^k_V\|_{\dot{H}^s\to \dot{H}^s}<1$, we can treat
 $(\I-\S^k_V)^{-1}$ by Neumann series to deduce  that it is a bounded operator on $\dot{H}^{s}$. This yields a solution $u_{k,x}\equiv e^{i{\psi}}\big(1+w\big)$ where
\begin{equation}\label{bukhsoln}
w\equiv w_{k,x}=(\I-\S^k_V)^{-1}\S^k_V[1]\in \dot{H}^s.
\end{equation}

In what remains of this section, we prove that $\S^k_V$ is contractive for sufficiently large $k$. This property will be crucial in the proof of Theorem~\ref{far} as well as in Section~\ref{recovery}.
We write $\S^k_V[f]=\tfrac{1}{4}\S^k_1[Vf]$, where
$$\S^k_1=
\partial_{\overline{z}}^{-1}\circ \M^{-{k}}\circ\partial_z^{-1} \circ
\M^{{k}}$$ and the multiplier operators $\M^{\pm{k}}$ are defined by
$
\M^{\pm{k}}[F]=
 e^{\pm i({\psi}+\overline{{\psi}})}\chi_{Q}\, F.
$
The key ingredient in the proof of the following estimate, is the classical lemma of van der Corput \cite{van}. 

\begin{lem}\label{isitgood} Let $0\le s_1,s_2<1$. Then
\begin{equation*}\|\M^{\pm k}[F](\cdot,x)\|_{\dot{H}^{-s_2}}\le Ck^{-\min\{s_1,s_2\}}\|F(\cdot,x)\|_{\dot{H}^{s_1}},\quad x\in \D,\ k\ge 1.
\end{equation*}
\end{lem}

\begin{proof} 
By the H\"older and Hardy--Littlewood--Sobolev inequalities, we have
\begin{equation}\label{ds}
\|\M^{\pm k}[F]\|_{2}\le C\|F\|_{\dot{H}^{s_1}},
\end{equation}
and 
\begin{equation}\label{fs}
\|\M^{\pm k}[F]\|_{\dot{H}^{-s_2}}\le C\|F\|_{2},
\end{equation}
with $0\le s_1,s_2<1$.
So by complex interpolation, it will suffice to prove that
\begin{equation}\label{rs}\|\M^{\pm k}[F]\|_{\dot{H}^{-s}}\le Ck^{-s}\|F\|_{\dot{H}^{s}}.
\end{equation}
Indeed, if $s_2<s_1$ we interpolate with \eqref{ds}, taking $s=s_1$, and if $s_1<s_2$ we interpolate with \eqref{fs}, taking $s=s_2$.
Now by real interpolation with the trivial $L^2$ bound, \eqref{rs} would follow from
\begin{equation}\label{besov}
\|\M^{\pm k}F\|_{\dot{B}^{-1}_{2,\infty}}\le Ck^{-1}\,\|F\|_{\dot{B}^{1}_{2,1}}
\end{equation}
(see Theorem 6.4.5 in \cite{BL}), where the Besov norms are defined as usual by
$$
\|f\|_{\bi}=\sup_{j\in\Z} 2^{-j}\|\P_jf\|_{L^2}\quad\quad \text{and}\quad\quad \|f\|_{\b}=\sum_{j\in\Z} 2^{j}\|\P_jf\|_{L^2}.
$$
Here, $\widehat{\P_j f}= \vartheta(2^{-j} |\cdot| ) \widehat{f}$  with $\vartheta$ satisfying $\text{supp}\, \vartheta \subset (1/2,2)$ and
$$
\sum_{j\in\Z} \vartheta(2^{-j} \cdot )=1.
$$
As
$
\|F\|_{\dot{B}^{-1}_{2,\infty}}\le C\|\widehat{F}\|_\infty
$
and
$
\|\widehat{F}\,\|_1\le C\|F\|_{\dot{B}^1_{2,1}},
$
the estimate \eqref{besov} would in turn follow from
\begin{equation}\label{pop}
\|\widehat{\M^{\pm k}F}\|_{\infty}\le Ck^{-1}\,\|\widehat{F}\|_{1}.
\end{equation}
Now, by the Fourier inversion formula and Fubini's theorem,
\begin{align*}
|\widehat{\M^{\pm k}F}(\xi)|&=\frac{1}{(2\pi)^2}\Big|\int_Q e^{\pm i(\psi(z)+\overline{\psi(z}))}\int \widehat{F}(\omega)\,e^{iz\cdot \omega} d\omega\, e^{-iz\cdot \xi} dz\Big|\\
&\le \int\Big|\int_Q
e^{\pm ik\frac{(z_1-x_1)^2-(z_2-x_2)^2}{4}}\,  e^{iz\cdot (\omega-\xi)}
dz\Big| |\widehat{F}(\omega)|\, d\omega
\end{align*}
so that \eqref{pop} follows by two applications of van der Corput's lemma \cite{van} (factorising the integral into the product of two integrals).
\end{proof}

In the following lemma, we optimise the decay in $k$, which will be important in Section~\ref{recovery}.

\begin{lem}\label{celtic7} Let $0<s<1$. Then
$$
\|\S^k_1[F](\cdot,x)\|_{\dot{H}^{s}}\le Ck^{-1}\|F(\cdot,x)\|_{\dot{H}^{s}},\quad x\in \D,\ k\ge 1.
$$
\end{lem}

\begin{proof} By two applications of Lemma~\ref{isitgood},
\begin{align*}
\|\S^k_1\|_{\dot{H}^{s}\to \dot{H}^{s}}
\le \|\M^{-{k}}\circ\partial^{-1}_z \circ \M^{{k}}\|_{\dot{H}^{s}\to \dot{H}^{s-1}}&\le Ck^{s-1}\|\partial^{-1}_z \circ \M^{{k}}\|_{\dot{H}^{s}\to \dot{H}^{1-s}}\\
&\le Ck^{s-1}\|\M^{{k}}\|_{\dot{H}^{s}\to \dot{H}^{-s}}\\
&\le Ck^{s-1-s}=Ck^{-1},
\end{align*}
and we are done.
\end{proof}

In the following lemma, we use Lemma~\ref{isitgood} only once, and
gain some integrability using the Hardy--Littlewood--Sobolev
theorem. By taking $k$ sufficiently large, we obtain our contraction
and thus our Bukhgeim solution $u=u_{k,x}$ as described above.

\begin{lem}\label{celtic3} Let $0<s<1$. Then
$$
\|\S^k_V[F](\cdot,x)\|_{\dot{H}^s}\le Ck^{-\min\{2s,1-s\}}\|V\|_{\dot{H}^{s}}\|F(\cdot,x)\|_{\dot{H}^s},\quad x\in \D,\ k\ge1.
$$
\end{lem}

\begin{proof} By the Cauchy--Schwarz and Hardy--Littlewood--Sobolev inequalities,
$$
\|VF\|_q\le \|V\|_{2q}\|F\|_{2q}\le C\|V\|_{\dot{H}^s}\|F\|_{\dot{H}^s},
$$
where $q=\frac{1}{1-s}$. Thus, as $\S^k_V[F]=\S^k_1[VF]$, it will suffice to prove that
$$
\|\S^k_1\|_{L^q\to \dot{H}^{s}}\le Ck^{{-\min\{2s,1-s\}}}.
$$
When $0<s<1/3$, by Lemma~\ref{isitgood}, we have
\begin{align*}
\|\S^k_1\|_{L^q\to \dot{H}^{s}}
\le \|\M^{-{k}}\circ\partial^{-1}_z \circ \M^{{k}}\|_{L^q\to \dot{H}^{s-1}}
&\le Ck^{-2s}\|\partial^{-1}_z \circ \M^{{k}}\|_{L^q\to \dot{H}^{2s}}\\
&\le Ck^{-2s}\|\M^{{k}}\|_{L^q\to \dot{H}^{2s-1}}\\
&\le Ck^{-2s}\|\M^{{k}}\|_{L^q\to L^q}.
\end{align*}
When $s\ge 1/3$, we also use H\"older's inequality at the end;
\begin{align*}
\|\S^k_1\|_{L^q\to \dot{H}^{s}}
\le \|\M^{-{k}}\circ\partial^{-1}_z \circ \M^{{k}}\|_{L^q\to \dot{H}^{s-1}}&\le Ck^{1-s}\|\partial^{-1}_z \circ \M^{{k}}\|_{L^q\to \dot{H}^{1-s}}\\
&\le Ck^{1-s}\|\M^{{k}}\|_{L^{q}\to \dot{H}^{-s}}\\
&\le Ck^{1-s}\|\M^{{k}}\|_{L^q\to L^{q^*}},
\end{align*}
where $q^*=\frac{2}{s+1}$, and so we are done.
 \end{proof}
 
 \begin{rem}
 Note that van der Corput's estimate is independent of the size of $Q$ and so, when $s\ge 1/3$, the potential need not be compactly supported for the results of this section to hold (when $s<1/3$ we used the compact support in a less obviously removable way).
 \end{rem}

\section{Proof of Theorem~\ref{far}}\label{fa}
In this section we show that the boundary values of our Bukhgeim solution
$u_{k,x}$  can be determined from knowledge of $\Lambda_V$. The argument  is inspired by \cite[Theorem 5]{Nachman96} but we
replace the Faddeev green function $G_k$ by its analogue in terms of
the operator $\S^k_V$ and avoid the use of single layer potentials.

Indeed, considering the kernel representation of $\S^k_1$, we can write $\S^k_V[F]$  in the form
\[ \begin{aligned}   \S^k_V[F](z)= \int_{\Omega} g_\psi (z,\eta) V(\eta)F(\eta)\, d\eta.  \end{aligned}
\]
where  $g_\psi$, the kernel of $\S^k_1$, is given by
\[ g_\psi(z,\eta)= \chi_{Q}(\eta)\frac{e^{i\big({\psi(\eta)}+\overline{{\psi(\eta}})\big)}}{4\pi^2}
 \int_{Q} \frac{1}{(\overline{\omega-\eta})(z-\omega)} e^{-i\big({\psi(\omega)}+\overline{{\psi(\omega}})\big)}\, d\omega.  \]
In order to work directly with exponentially growing solutions we
conjugate~$g_\psi $ with the exponential factors, so that
\begin{equation}\label{its}
\int_{\Omega} G_\psi (z,\eta) V(\eta)  F(\eta)\, d\eta= e^{i
{\psi(z)}} \S^k_V[e^{-i{\psi}}  F](z),
\end{equation}
where $G_\psi(z,\eta)=
e^{i {\psi}(z)} g_\psi(z,\eta)e^{-i {\psi} (\eta)}$.
Notice also that when $z   \in Q\backslash \D$ and
$\eta \in \D$, we have that
\begin{equation*} \label{harmonic}
\Delta_\eta G_\psi(z,\eta) =0.
\end{equation*}
Thus, if we take \eqref{its} with $F=P_V(f)$, where $P_V(f)$ solves
$\Delta u=Vu$ with $u|_{\partial\Omega}=f$, using Alessandrini's
identity we obtain that, for each $z \in Q \setminus \D$,
\begin{equation} \label{its2}
 \Big\langle   (\Lambda_V-\Lambda_0)[f], G_\psi(z,\cdot)|_{\partial\D}\Big\rangle=e^{i {\psi(z)}} \S^k_V[e^{-i{\psi}}  P_V(f)](z).
\end{equation}
In particular  the right-hand side belongs to $H^1(Q\setminus\D)$
and hence we can define the operator  $\Gamma_{\!\psi}: H^{1/2} \to
H^{1/2}$  by
\[ \Gamma_{\!\psi}[f]=  T_r\circ \Big\langle   (\Lambda_V-\Lambda_0)[f], G_\psi|_{\partial\D} \Big\rangle, \]
where $T_r : H^1(Q\setminus \D)\to H^{1/2}(\partial\D)$ is the trace
operator. Now, by the definitions of $u_{k,x}$ and $w$, we also
deduce from \eqref{its} and \eqref{bukhsoln} that
\begin{equation}\label{its3}
\int_{\Omega} G_\psi (\cdot,\eta) V(\eta)\, u_{k,x}(\eta)\, d\eta=
e^{i {\psi}} \S^k_V[1+w]= e^{i {\psi} } w= u_{k,x}-e^{i{\psi}}.
\end{equation}
Combining \eqref{its}, \eqref{its2}, and \eqref{its3} we obtain the integral identity
\begin{equation*} (\I-\Gamma_{\!\psi})[u_{k,x}|_{\partial\D}]=e^{i{\psi}}|_{\partial\D}. \end{equation*}

Thus, we can determine $u_{k,x}$ on the boundary if we can invert
$(\I- \Gamma_{\!\psi})$. By the Fredholm alternative it will suffice
to show that $\Gamma_{\!\psi}$ is compact and that
$(\I-\Gamma_{\!\psi})$ has a trivial kernel on $H^{1/2}(\partial
\D)$.

\begin{thm} Let $V \in \dot{H}^s$ with $0<s<1$. Then
\begin{itemize}
\item [{(i)}] $\Gamma_{\!\psi}$ is compact
\item [{(ii)}] $\Gamma_{\!\psi}[f]=f \Rightarrow f=0$.
\end{itemize}
\end{thm}

\begin{proof}[Proof of (i)]

We have that
\[\Gamma_{\!\psi}[f]= T_r\big[ e^{i{\psi}} \S^k_V[e^{-i{\psi}}P_V(f)]\big]. \]
As the set of compact operators is a left and right ideal, we
consider the boundedness properties of each component of the
composition. Firstly,  $P_V:H^{1/2}(\partial\D) \to H^1(\D)$ is
bounded.  Secondly, $H^1(\D) \hookrightarrow   L^p(\Omega)$ compactly for
all $2<p<\infty$. Now taking $p$ sufficiently large and
$\tfrac{1}{2}=\tfrac1{q}+\tfrac1p,$ by the boundedness of the Cauchy
transform followed by the Hardy--Littlewood--Sobolev inequality,
\begin{align*}
\|\S^k_V[e^{-i{\psi}}G]\|_{H^1(Q\setminus \D)}\le C\|VG\|_{L^2(\Omega)}&\le
C\|V\|_{L^q(\Omega)}\|G\|_{L^p(\Omega)}\\
&\le C\|V\|_{\dot{H}^s}\|G\|_{L^p(\Omega)}.
\end{align*}
Finally, $T_r:H^1(Q\setminus \D) \to H^{1/2}(\partial\D)$ is
bounded. Since the embedding $H^1(\D) \hookrightarrow L^p(\Omega)$
is compact,  it follows that $\Gamma_{\!\psi}$ is compact.
\vspace{0.5em}

\noindent{\it Proof of (ii).}  Letting $\rho= \S^k_V[e^{-i{\psi}}
P_V(f)]$, we have that $$\dc [e^{i{\psi}} {\rho}] =
\tfrac{1}{4}e^{-i\overline{{\psi}}}\chi_Q \d^{-1}[ e^{i\overline{{\psi}}} V
P_V(f)],$$ so that $$4\d \dc [e^{i{\psi}} {\rho}] = V P_V(f) \quad
\mbox{ on }\, \D.$$ This can be rewritten as $ \Delta[ e^{i{\psi}}
{\rho} - P_V(f)] = 0$ on $\Omega$.  Now by hypothesis
$\Gamma_{\!\psi}[f]=f$, so that by \eqref{its2} we have $e^{i{\psi}}
{\rho}=f$ on $\partial \D$. Combining the two, we see that
$$e^{i{\psi}} {\rho} = P_V(f) \quad
\mbox{ on }\, \D.$$  From the definition of $\rho$ we see that
${\rho} = \S^k_V[{\rho}]$, and as soon as $\S^k_V$ is strictly
contractive, that ${\rho}=0$. This of course follows from
Lemma~\ref{celtic3} for large enough $k$. Thus, $f = e^{i{\psi}}
{\rho} = 0$, so that $\I-\Gamma_{\!\psi}$ is injective as desired.
\end{proof}

\begin{rem}\label{rem} We need not suppose that the potential is  compactly supported here as long as we suppose that $\chi_\Omega V\in H^\varepsilon$ and then the Bukhgeim solutions which we identify are associated to this potential instead. For  $0<\varepsilon<1/2$ and $\Omega$ Lipschitz,  we have  $\chi_\Omega V\in H^\varepsilon$ as long as $V\in H^{s}$ with $s>1/2+\varepsilon$. To see this, note that by the fractional Leibnitz rule (see for example \cite{KPV}), 
$$
\|\chi_\Omega V\|_{H^\varepsilon}\le \|\chi_\Omega\|_4\|V\|_{W^{\varepsilon,4}}+\|\chi_\Omega\|_{W^{\varepsilon,p}}\|V\|_{L^q}
$$
with $p<\frac{4}{1+2\varepsilon}$ and $\frac{1}{p}+\frac{1}{q}=\frac{1}{2}$. Then the remark follows by the Hardy--Littlewood--Sobolev inequality, combined with the fact that $\chi_{\Omega}\in H^s$ for all $s<1/2$ (see for example \cite{FR}). 
\end{rem}

\section{Potential recovery}\label{recovery}

In order to recover the potential at $x\in \D$, it remains to show
that the right-hand side of Alessandrini's identity \eqref{ale} converges to $V(x)$.
That is to say $\T^k_{1+w} V(x)$ converges to $V(x)$ as $k$ tends to infinity, where
$$
\T^k_{1+w}[F](x)=\frac{k}{4\pi}\int_{\R^2}
e^{i({\psi(z)}+\overline{{\psi}(z}))}\,F(z)\big(1+w(z)\big)\,dz.
$$
First we show that $\T_w^kV$ can be considered to be a remainder term.

\begin{thm}\label{Remainder}
Let $V  \in \dot{H}^{s}$ with $0<s<1$. Then
\[ \lim_{k \to \infty}   \T^k_w[V](x)=0,\quad x\in\D.\]
Moreover, if  $k\ge (1+c\|V\|_{\dot{H}^{s}})^{\max\{\frac{1}{2s},\frac{1}{1-s}\}}$, then
$$
\sup_{x\in\D}|\T^k_w[V](x)|\le Ck^{-s}\|V\|_{\dot{H}^{s}}^2.
$$
\end{thm}

\begin{proof}
By Lemma~\ref{isitgood},
\[ \begin{aligned} |\T^k_w[V](x)| & \le  Ck  \| \M^{{k}} [V]\|_{\dot{H}^{-s}}\|w\|_{\dot{H}^{s}}
\\  & \le  Ck^{1-s}  \|V\|_{\dot{H}^{s}}  \|(\I-\S^k_V)^{-1}\S^k_V[1]\|_{\dot{H}^{s}}.
\end{aligned}\]
By Lemma \ref{celtic3},  we can treat
 $(\I-\S^k_V)^{-1}$ by Neumann series to deduce  that it is a bounded operator on $\dot{H}^{s}$ when $k\ge 1$ and  $Ck^{-\min\{2s,1-s\}}\|V\|_{\dot{H}^{s}}\le \frac{1}{2}$. Then
\[\begin{aligned}
|\T^k_w[V](x)| & \le  Ck^{1-s}  \|V\|_{\dot{H}^{s}}  \|\S^k_1[V]\|_{\dot{H}^{s}}  \\
& \le   Ck^{-s}  \|V\|^2_{\dot{H}^{s}},
\end{aligned} \]
by an application of Lemma~\ref{celtic7}, which is the desired estimate.
\end{proof}

Noting that $e^{i({\psi(z)}+\overline{{\psi(z}}))}=
\exp\big(ik\frac{(z_1-x_1)^2-(z_2-x_2)^2}{4}\big)$, it remains to prove
\begin{equation}\label{ipo}
\lim_{k\to \infty} \T^k_{1}[V](x)=V(x),
\end{equation}
where $\T_1^{k}$ is defined by
\begin{equation*}\label{no}
\T_1^{k}[F](x)=\frac{k}{4\pi}\int
\exp\big(ik\tfrac{(z_1-x_1)^2-(z_2-x_2)^2}{4}\big) F(z)\,dz.
\end{equation*}
Now when $F$ is a Schwartz function, this is equal to
$e^{i\frac{1}{k}\Box}F(x)$, where
\begin{equation*}\label{nonon}
e^{i\frac{1}{k}\Box}[F](x)=\frac{1}{(2\pi)^2}\int_{\R^2}e^{ix\cdot
\xi}\, e^{-i\frac{1}{k}(\xi_1^2-\xi_2^2)} \,\widehat{F}(\xi)\, d\xi.
\end{equation*}
This follows easily, making use of the distributional formula
$$
\frac{k}{4\pi}\int e^{ik\frac{z_1^2-z_2^2}{4}} \phi(z)\, dz =\int e^{-i\frac{1}{k}(\xi_1^2-\xi_2^2)} \widehat{\phi} (\xi)\, d\xi,
$$
which holds for Schwartz functions $\phi$.  We see that when $V$ is
a Schwartz function, $\T^k_{1}V$ solves the time-dependent
nonelliptic Schr\"odinger equation, $$i\partial_t u +\Box u=0,$$ where
$\Box =\partial_{x_1x_1}-\partial_{x_2x_2}$, with initial data $V$
at time $1/k$. When $V\in H^s$ with $s>1$, both $V$ and its Fourier
transform are integrable, and so both $\T^k_{1}V$ and
$e^{i\frac{1}{k}\Box}V$ are continuous functions which are again
equal pointwise. Thus, in the following lemma we obtain the convergence \eqref{ipo} and therefore complete the reconstruction
for potentials in $H^s$ with $s>1$.

\begin{lem}\label{balls} Let $V\in H^s$ with $1<s<3$.  Then
$$|e^{i\frac{1}{k} \square }V(x)- V(x)|\le Ck^{\frac{1-s}{2}}\|V\|_{H^{s}},\quad x\in \Omega.$$
\end{lem}

\begin{proof}
By the Fourier inversion formula and the Cauchy--Schwarz inequality,
\begin{align*}
|e^{it \square }V(x)- V(x)|&=\frac{1}{(2\pi)^2}\Big|\int \widehat{V}(\xi)\, e^{i\xi\cdot x}\big( e^{-i\frac{1}{k}(\xi^2_1-\xi_2^2)}-1\big)\, d\xi\Big|\\
&\le \|V\|_{H^s} \Big(\int\frac{|e^{-i\frac{1}{k}(\xi^2_1-\xi_2^2)}-1|^2}{|\xi|^{2s}}d\xi\Big)^{1/2}\\
&= \|V\|_{H^s}  \Big(\int\frac{2-2\cos\big(\frac{1}{k}(\xi^2_1-\xi_2^2)\big)}{|\xi|^{2s}}d\xi\Big)^{1/2}\\
&= 2k^{\frac{1-s}{2}}\|V\|_{H^s}  \Big(\int\frac{\sin^2\big(\tfrac{1}{2}(\xi^2_1-\xi_2^2)\big)}{|\xi|^{2s}}d\xi\Big)^{1/2}\\
&\le 2k^{\frac{1-s}{2}}\|V\|_{H^s}
\Big(\int_{\mathbb{D}}\frac{1}{|\xi|^{2(s-2)}}d\xi+\int_{\R^2\backslash\mathbb{D}}\frac{1}{|\xi|^{2s}}\Big)^{1/2},
\end{align*}
where we have used the  trigonometric identity
$2\sin^2\theta=1-\cos2\theta$ and the fact that $\sin \theta\le
|\theta|$.
\end{proof}

Altogether we see that $|\T^k_{1+w} V(x)-V(x)|\le Ck^\frac{1-s}{2}$ for all $x\in \Omega$ and $V\in H^s$ with $1<s<3$,  which improves upon the decay rate of \cite{NS} where they recovered $C^2$ potentials. Note that there can be no decay rates, at least for the main term, for the potentials of $H^s$ with $s\le 1$ as they would then be uniform limits of continuous functions and thus continuous.

For discontinuous potentials we are no longer able to recover at each point. Instead we bound the fractal dimension of the sets where the recovery fails. This point of view has its origins in the work of Beurling  who bounded the capacity of the divergence sets of Fourier series \cite{beur} (see also \cite{BBCR}). Now Sobolev spaces are only defined modulo sets of zero Lebesgue measure, and so we consider first the potential spaces
$$L^{s,2}=\{\,I_s\ast g\,:\,g\in L^2(\R^2)\,\},$$ where $I_s$
is the Riesz potential $|\cdot|^{s-2}$. As
 $\widehat{I}_{\!s}(\xi)=C_s|\xi|^{-s}$, we have that $I_s\ast g$ is also a member of (an equivalence class  of) $H^s$.

To bound the dimension of the sets where the recovery fails, we will prove maximal estimates with respect to fractal measures.
We say that a positive Borel measure $\mu$ is $\alpha$-dimensional
if
\begin{equation}\label{def3}
c_{\alpha}(\mu):=\sup_{x\in\mathbb{R}^2,\,r>0}\frac{\mu\big(B(x,r)\big)}{r^{\alpha}}<\infty,\quad\quad
0\leq\alpha\leq 2,
\end{equation}
and denote by $\mathcal{M}^\alpha(\D)$ the
$\alpha$-dimensional probability measures which are supported in $\Omega$.
For $0<s<1$, we will require the elementary inequality
 \begin{equation}\label{element0}
\|I_s\ast g\|_{L^1(d\mu)}\ls \sqrt{c_\alpha(\mu)}\,\|g\|_{L^2(\R^2)},\quad \alpha>2-2s,
\end{equation}
which holds whenever $\mu\in \mathcal{M}^\alpha(\D)$ and $g\in L^2(\R^2)$. To see this, we note that by Fubini's
theorem and the Cauchy--Schwarz inequality,
\begin{align*}
\|I_{s}*g\|_{L^1(d\mu)}
&\le  \|I_s\ast \mu\|_{L^2}\|g\|_{L^2},
\end{align*}
so that \eqref{element0} follows by proving
$$
\|I_s\ast \mu\|^2_{L^2}\ls\, c_\alpha(\mu),\quad \alpha>2-2s.
$$
Now by Plancherel's theorem,
\begin{align*}
\|I_s\ast
\mu\|^2_{L^2}=(2\pi)^{-2}\|\widehat{I}_s\widehat{\mu}\|^2_{L^2}\ls\int \widehat{\mu}(\xi)\,\overline{\widehat{\mu}(\xi)}\,\widehat{I}_{2s}(\xi)\,d\xi&\ls\int \mu\ast I_{2s}(y)\,d\mu(y)\\
&= \int\!\! \int \frac{d\mu(x)d\mu(y)}{|x-y|^{2-2s}},
\end{align*}

which is nothing more than the $(2-2s)$-energy. Then, by an appropriate
dyadic decomposition,
\begin{align*}\label{energy}
\int\!\! \int\frac{d\mu(x)d\mu(y)}{|x-y|^{2-2s}}&\ls \int
\sum_{j=0}^\infty c_\alpha(\mu) 2^{-j\alpha}2^{j(2-2s)} d\mu(y)\ls\, c_\alpha(\mu)
\end{align*}
whenever  $\alpha>2-2s$ and $\mu\in\mathcal{M}^\alpha(\D)$.

The Fourier transform of less regular potentials $V$ is not
necessarily integrable,  and so in that case $e^{i\frac{1}{k}\Box}V$
is not even well-defined. Instead we make do with the pointwise
limit
\begin{equation}\label{pointlim}
\T^k_1[V](x)=\lim_{N\to \infty} G_N\ast \T^k_1[V](x)=\lim_{N\to \infty} e^{i\frac{1}{k}\Box}[G_N\ast V](x),\quad x\in \D,
\end{equation}
where $G_N=N^{2}G(N\cdot)$ and $G$ is the Gaussian $e^{-|\cdot|^2}$. This formula holds as~$V$ is compactly supported and integrable; conditions which the initial data in the time-dependent theory does not normally satisfy.
We will also
require the following lemma due, in this form,  to Sj\"olin
\cite{Sj0}.

\begin{lem}\label{lem2}\cite{Sj0}
Let  $x,t\in\R$, $\gamma\in [1/2,1)$ and $N\geq 1$. Then
$$
\left|\int_\R \frac{\eta(N^{-1}\xi)\,e^{
i(x\xi-t\xi^2)}}{|\xi|^{\gamma}}\,d\xi\right| \ls
\frac{1}{|x|^{1-\gamma}},
$$
where the constant implied by the symbol $\ls$ depends only on $\gamma$ and the
Schwartz function $\eta$.
\end{lem}
In the following theorem, we employ the
Kolmogorov--Seliverstov--Plessner method, as used by
Carleson~\cite{C} for the one-dimensional Schr\"odinger equation.
Dahlberg and Kenig~\cite{DK} proved that the result of Carleson is
sharp and noted that his argument could be applied to the higher
dimensional problem (for which the argument is no longer sharp for the elliptic equation, see \cite{Bo}). We refine their argument, which extends to the nonelliptic case, by proving estimates which hold uniformly with respect to fractal measures.

\begin{thm}\label{it} Let $1/2\le s<1$. Then
$$
\big\|\,\sup_{k\ge 1}\sup_{N\ge 1} |e^{i\frac{1}{k}\Box}[G_N\ast I_s\ast g]|\,\big\|_{L^1(d\mu)} \ls \sqrt{c_\alpha(\mu)} \|g\|_{L^2(\R^2)}, \quad \alpha>2-s,
$$
whenever $\mu\in \mathcal{M}^\alpha(\D)$ and $g\in L^2$. 
\end{thm}

\begin{proof}
By linearising, it will suffice to prove
\begin{equation}\label{linenergy}
\left|\int_{\D}e^{it(x)\Box}[G_{N(x)}\ast I_s\ast g]\,w(x)\,
d\mu(x)\right|^2\ls c_\alpha(\mu)\,\|g\|^2_{L^2},\quad \alpha >2-s,
\end{equation}
uniformly in measurable functions $t:\D\to\R$, $N:\D\to \mathbb{N}$
and $w:\D\to \mathbb{D}$. By Fubini's theorem and the
Cauchy--Schwarz inequality, the left-hand side of \eqref{linenergy}
is bounded by
\begin{equation*}
\int |\widehat{g}(\xi)|^2d\xi \int \left|\int
G\Big(\frac{\xi}{N(x)}\Big)\,e^{it(x)(\xi_1^2-\xi_2^2)}e^{ ix\cdot\xi}w(x)\, d\mu(x)
\right|^2\frac{d\xi}{|\xi|^{2s}}.
\end{equation*}
  Writing the squared integral as a double
integral, and applying Fubini's theorem again, it will suffice to show~that
\begin{align}
\!\!\!\!\int\!\!\!\!\int\!\!\!\!\int
\!G\Big(\frac{\xi}{N(x)}\Big)\,G\Big(\frac{\xi}{N(y)}\Big)\,e^{i(t(x)-t(y))(\xi_1^2-\xi_2^2)}&e^{
i(x-y)\cdot\xi}\frac{d\xi}{|\xi|^{2s}}\times\nonumber\\
&\, w(x)w(y)\,d\mu(x)d\mu(y)\ls
c_\alpha(\mu)\label{polpo}\end{align} uniformly in the functions
$t$, $N$ and $w$. Now, as $|\xi|^{2s}\ge|\xi_1|^{s}|\xi_2|^{s}$, the
left-hand side of \eqref{polpo} is bounded by
\begin{align*}
\prod_{j=1}^2\Big|\int
G\Big(\frac{\xi_j}{N(x)}\Big)\,G\Big(\frac{\xi_j}{N(y)}\Big)\,e^{i(-1)^{j+1}(t(x)-t(y))\xi_j^2}
&e^{i(x_j-y_j)\xi_j}\frac{d\xi_j}{|\xi_j|^{s}}\Big|\times\\
&\, w(x)w(y)\,d\mu(x)d\mu(y),
\end{align*}
and by Lemma~\ref{lem2}, we have
$$
\left|\int \!\frac{G\Big(\frac{\xi_j}{N(x)}\Big)\,G\Big(\frac{\xi_j}{N(y)}\Big)\,e^{
i(-1)^{j+1}(t(x)-t(y))\xi_j^2}e^{
i(x_j-y_j)\xi_j}}{|\xi_j|^{s}}d\xi_j\right|\ls\frac{1}{|x_j-y_j|^{1-s}}.
$$
Substituting in, we see that the left-hand side of (\ref{polpo}) is
bounded by
\begin{align}\label{itt}
C\!\int\!\!
\int\frac{|w(x)w(y)|d\mu(x)d\mu(y)}{|x_1-y_1|^{1-s}|x_2-y_2|^{1-s}}&\le C\!\int\!\!
\int\frac{d\mu(x)d\mu(y)}{|x_1-y_1|^{1-s}|x_2-y_2|^{1-s}}
\end{align}

To complete the proof, we are required to bound \eqref{itt} by
$c_\alpha(\mu)$. This will require a dyadic
decomposition which lends itself to the singularities along the axis-parallel lines $A_y$ defined by
$$
A_y=\{ x\in\D\,:\, x_1=y_1\quad \text{or}\quad x_2=y_2\big\},\quad y\in\D.
$$
Covering $A_y$ by balls $\{B_j\}_{j\ge1}$ of radius $r_j$ and using the definition \eqref{def3} of $c_\alpha(\mu)$, we have
$$
\mu(A_y)\le \sum_{j\ge 1} \mu (B_j)\le c_{\alpha}(\mu) \sum_{j\ge 1} r_j^\alpha.
$$
Taking the infimum over all such coverings and using the fact that the
$\alpha$-Hausdorff measure of $A_y$ is zero when $\alpha>1$, we see that $\mu(A_y)=0$ for all $\mu \in \mathcal{M}^\alpha(\D)$. Thus we can ignore the sets $A_y$ when decomposing the inner integral of  \eqref{itt}.

For each $j,\ell\in\mathbb{Z}$ we break up $Q\supset\Omega$ into
dyadic rectangles of dimensions
$2^{-j}\times2^{-\ell}$ and consider the unique rectangle $R_{j,\ell}$ which contains $y$. We call the unique
rectangles $R_{j-1,\ell-1}, \,\,R_{j-1,\ell},$ and
$R_{j,\ell-1}$ that contain $R_{j,\ell}$, the mother, the
father, and the stepfather respectively. We write $R_{j,\ell}^n\sim
R_{j,\ell}$ if their mothers touch, but their fathers and
stepfathers do not. As $\mu(A_y)=0$,  we can write
 $$
 \int F(x,y)\,d\mu(x)=\sum_{j,\ell\ge 0}\sum_{\ n: R_{j,\ell}^{n}\sim R_{j,\ell}}\int_{R_{j,\ell}^{n}} F(x,y)\,d\mu(x),
 $$
\[\begin{picture}(128,128)
\put(0,0){\line(0,1){128}} \put(0,0){\line(1,0){128}}
\put(128,0){\line(0,1){128}}\put(0,128){\line(1,0){128}}
\put(48.5,51){$\cdot$} \put(53,52){$y$}
\put(105,108){$R^{n}_{j,\ell}$}

\put(0,16){\line(1,0){44}}
\put(0,32){\line(1,0){44}}
\put(0,40){\line(1,0){44}}
\put(0,44){\line(1,0){44}}
\put(0,48){\line(1,0){44}}
\put(0,60){\line(1,0){44}}
\put(0,64){\line(1,0){44}}
\put(0,72){\line(1,0){44}}
\put(0,80){\line(1,0){44}}
\put(0,96){\line(1,0){44}}
\put(0,128){\line(1,0){44}}

\put(56,16){\line(1,0){72}}
\put(56,32){\line(1,0){72}}
\put(56,40){\line(1,0){72}}
\put(56,44){\line(1,0){72}}
\put(56,48){\line(1,0){72}}
\put(56,60){\line(1,0){72}}
\put(56,64){\line(1,0){72}}
\put(56,72){\line(1,0){72}}
\put(56,80){\line(1,0){72}}
\put(56,96){\line(1,0){72}}
\put(56,128){\line(1,0){72}}

\put(16,0){\line(0,1){48}}
\put(32,0){\line(0,1){48}}
\put(40,0){\line(0,1){48}}
\put(44,0){\line(0,1){48}}
\put(56,0){\line(0,1){48}}
\put(60,0){\line(0,1){48}}
\put(64,0){\line(0,1){48}}
\put(72,0){\line(0,1){48}}
\put(80,0){\line(0,1){48}}
\put(96,0){\line(0,1){48}}
\put(128,0){\line(0,1){48}}

\put(16,60){\line(0,1){68}}
\put(32,60){\line(0,1){68}}
\put(40,60){\line(0,1){68}}
\put(44,60){\line(0,1){68}}
\put(56,60){\line(0,1){68}}
\put(60,60){\line(0,1){68}}
\put(64,60){\line(0,1){68}}
\put(72,60){\line(0,1){68}}
\put(80,60){\line(0,1){68}}
\put(96,60){\line(0,1){68}}
\put(128,60){\line(0,1){68}}
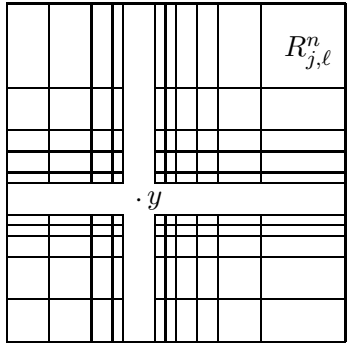
\put(-47.5,-18){{\tiny \bf The rectangles of dimensions $2^{-j}\times 2^{-\ell}$, with $1\le j,\ell\le 3$,}}
\put(1.5,-28){{\tiny \bf associated with a single point $y$.}}
\end{picture}
\vspace{2em}
\]
which yields
$$
\eqref{itt}\ \le\, C\!\int \sum_{j,\ell\ge 0}\sum_{\ n: R_{j,\ell}^{n}\sim R_{j,\ell}}2^{j(1-s)}
2^{\ell(1-s)}\mu(R_{j,\ell}^{n})\,d\mu(y).
$$
Without loss of generality, we can suppose that
\begin{align*}
\sum_{\ell>j\ge 0}\sum_{\ n: R_{j,\ell}^{n}\sim R_{j,\ell}}2^{j(1-s)} 2^{\ell(1-s)}\mu(R_{j,\ell}^{n})
\le \sum_{j\ge \ell\ge
0}\sum_{\ n: R_{j,\ell}^{n}\sim R_{j,\ell}}2^{j(1-s)} 2^{\ell(1-s)}\mu(R_{j,\ell}^{n}),
\end{align*}
so that
$$
\eqref{itt}\ \le\, C\!\int \sum_{j\ge \ell\ge
0}\sum_{\ n: R_{j,\ell}^{n}\sim R_{j,\ell}}2^{j(1-s)}
2^{\ell(1-s)}\mu(R_{j,\ell}^{n})\,d\mu(y).
$$
Now by covering each rectangle by discs of radius $2^{-j}$, and using the definition \eqref{def3} of $c_{\alpha}(\mu)$, we see that
$$
\mu(R_{j,\ell}^{n})\ \ls\  2^{j-\ell} c_{\alpha}(\mu) 2^{-j\alpha},
$$
and for each rectangle $R_{j,\ell}$ there are exactly nine
rectangles $R_{j,\ell}^{n}$ which satisfy $R_{j,\ell}^n\sim
R_{j,\ell}$. Thus
\begin{align*}
\eqref{itt}\ &\ls\  c_{\alpha}(\mu) \sum_{j\ge \ell\ge 0} 2^{j(2-s-\alpha)}2^{-\ell s}\ \ls\  c_{\alpha}(\mu),
\end{align*}
when $\alpha>2-s$, and so we are done. 
\end{proof}

\noindent{\it Proof of Theorem~\ref{dim2}.} By Alessandrini's identity \eqref{ale} and Frostman's lemma (see for example \cite{Mat}), it will suffice to prove that
\begin{equation}\label{zero}
\mu\Big\{ x \,:\,  \limsup_{k\to \infty} |\T^k_{1+w}[V](x)-V(x)|\neq 0\Big\}=0
\end{equation}
whenever $\mu\in \M^\alpha(\D)$ and $V\in L^{s,2}(\D)$ with
$\alpha>2-s$. By Theorem~\ref{Remainder} and~\eqref{pointlim}, this
would follow from
$$
\mu\Big\{ x \,:\,  \limsup_{k\to \infty} \limsup_{N\to \infty}|
e^{i\frac{1}{k}\Box}[G_N\ast V](x)-V(x)|\neq 0\Big\}=0.
$$
Writing $V=I_s\ast g$, where $g\in L^2$, we take a Schwartz function
$h$  so that $\|g-h\|_{L^2}<\epsilon$. Then
\begin{align*}
&\ \mu\Big\{\, x \,:\,  \limsup_{k\to \infty}\limsup_{N\to \infty} |e^{i\frac{1}{k}\Box} [G_N \ast V](x)-V(x)|>\la\,\Big\}\\
\le&\,\, \mu\Big\{\, x \,:\,  \sup_{k\ge1}\sup_{N\ge 1} |e^{i\frac{1}{k}\Box} [G_N \ast I_s\ast (g-h)](x)|>\la/3\,\Big\}\,+\\
&\,\, \mu\Big\{\, x \,:\,  \limsup_{k\to \infty}\limsup_{N\to \infty} |e^{i\frac{1}{k}\Box} [G_N \ast I_s\ast h](x)-I_s\ast h(x)|>\la/3\,\Big\}\,+\\
&\,\, \mu\Big\{\, x \,:\,  |I_s\ast (h-g)(x)|>\la/3\,\Big\}.
\end{align*}
As the terms involving $h$ are continuous in all parameters, the second set of the three is empty, so by the elementary inequality \eqref{element0} and Theorem~\ref{it}, we see that
\begin{align*}
\mu\Big\{ x \,:\,  \limsup_{k\to \infty}
|\T^k_{1+w}V(x)-V(x)|>\la\Big\}&\ls
\lambda^{-1}\sqrt{c_\alpha(\mu)}\,\|g-h\|_{L^2}\\
&\ls \lambda^{-1}\sqrt{c_\alpha(\mu)}\,\epsilon,
\end{align*}
for all $\epsilon>0$, which yields \eqref{zero}, and so we are done.\hfill$\Box$

\vspace{1em}

\noindent{\it Proof of Theorem~\ref{names}.} This follows by applying Corollary~\ref{dim} to the potential $q=V-\kappa^2\chi_\Omega$. For $V\in H^{1/2}$, the potentials
$q=V-\kappa^2\chi_\Omega$ are contained in  $\dot{H}^s$ for $0<s< 1/2$ (see for example \cite{FR}) and
so we find Bukhgeim  solutions $U_{k,x}$, associated to $q$,  and
recover their value on the boundary as before. However, Corollary~\ref{dim}
requires the potential $q$ to be contained in $H^{1/2}$ which is not
satisfied for any domain. However, it is clear from the proof of
Theorem~\ref{it} that we can relax this condition further to
$$\big\|\big(i\tfrac{\partial}{\partial_{x_1}}\big)^{1/4}\big(i\tfrac{\partial}{\partial_{x_2}}\big)^{1/4}q\big\|_{L^2(\R^2)}<\infty,$$
which is satisfied when $\Omega$ is a axis-parallel square, but not when it is a disc.~\hfill$\Box$

\begin{rem}
As in the previous sections we can consider potentials which are not compactly supported. Here we can recover the potentials on $\Omega$  if $V\in H^{s}$ with $s>3/4$. Indeed, the arguments of this section require that
$$\big\|\big(i\tfrac{\partial}{\partial_{x_1}}\big)^{1/4}\big(i\tfrac{\partial}{\partial_{x_2}}\big)^{1/4}(\chi_\Omega V)\big\|_{L^2(\R^2)}<\infty,$$
for which it is again convenient to take $\Omega$ to be an axis-parallel square. Then arguing as in Remark~\ref{rem}, by the fractional Leibnitz rule,
$$
\big\|\big(i\tfrac{\partial}{\partial_{x_2}}\big)^{1/4}(\chi_\Omega V)(x_1,\cdot)\big\|_{L^2(\R)}\le \|\chi_\Omega(x_1,\cdot)\|_4\|\big(i\tfrac{\partial}{\partial_{x_2}}\big)^{1/4}V(x_1,\cdot)\|_4
$$

 By factorising the integral using Fubini's theorem and applying the argument of Remark~\ref{rem} in the $x_2$-variable, this holds if 
$$\big\|\big(i\tfrac{\partial}{\partial_{x_1}}\big)^{1/4}\big(i\tfrac{\partial}{\partial_{x_2}}\big)^{s_0}V\big\|_{L^2(\R^2)}<\infty,$$
with $s_0>1/2$. Thus if a noncompactly supported potential is in $H^{s}$ with $s>3/4$, we can recover it on any  compact domain.
\end{rem}

\vspace{1em}

Finally we note that the uniqueness result of Bl\aa sten \cite{Blastenpp} can be observed using the connection with the time-dependent Schr\"odinger equation. Indeed if the scattering data or boundary measurements are the same for two potentials $V_1$ and $V_2$, then by Alessandrini's identity~\eqref{ale},
$$
\|V_2-V_1\|_{L^2}=\|V_2-\T^k_{1+w} V_2+\T^k_{1+w} V_1-V_1\|_{L^2},
$$
so that by the triangle inequality and Lemma~\ref{Remainder}, it suffices to prove 
$$
\|V-\T^k_{1} V\|_{L^2}\to 0\quad \textrm{as}\quad k\to \infty,
$$
which is a well-known property of the Schr\"odinger flow.

\section{Proof of Theorem \ref{sharpness}}\label{sharp}
First we construct a real potential $V$, supported in $\D$, and  contained in $H^s$ with $s<1/2$, for which
\begin{align*}
\Big|\Big\{ x \in\D\,:\, \lim_{k\to \infty}e^{i\frac{1}{k}\Box}[V](x) \not\to V(x)\Big\}\Big|\neq 0.
\end{align*}
Throughout this section we work with a different set of coordinates
from the previous sections. Indeed, for Schwartz functions $F$, we
now write
$$
e^{it\Box}[F](x)= \frac{1}{(2\pi)^2}\int e^{ix\cdot \xi}
e^{-i2t\xi_1\xi_2}\, \widehat{F}(\xi)\,   d\xi.
$$
Let $\phi_\text{o}$ be a positive Schwartz function, compactly
supported in $[-1/4,1/4]$, and consider $\phi=\phi_\text{o}\ast
\phi_\text{o}$, which is supported in $[-1/2,1/2]$. Note that
$\widehat{\phi}=(\widehat{\phi}_\text{o})^2\ge 0$. We consider the
potential $V$ defined by
\begin{align*}V(x)&=\sum_{j\ge 2}V_j(x)=\sum_{j\ge2} 2^{(1-\beta) j+1} \cos({2^j x_2}) \phi(2^{j}x_1)\phi(x_2)\\ &=   \sum_{j\ge2} 2^{(1-\beta) j}e^{i2^j x_2} \phi(2^{j}x_1)\phi(x_2)+\sum_{j\ge2} 2^{(1-\beta) j}e^{-i2^j x_2} \phi(2^{j}x_1)\phi(x_2)\\ &= \sum_{j\ge 2}V^+_j(x)+\sum_{j\ge
2}V^-_j(x),
\end{align*}
which is supported in
$[-\tfrac{1}{8},\tfrac{1}{8}]\times[-\tfrac{1}{2},\tfrac{1}{2}]$. If
$\beta\in (1/2+s,1)$, by changes of variables,
\begin{align*}
\|V\|^2_{H^s}
&\le C\sum_{j\ge 2} 2^{(1-2\beta+2s) j} \int |\widehat{\phi}(\xi_1)\widehat{\phi}(\xi_2)|^2\,(1+|\xi|^2)^{s}d\xi<\infty.
\end{align*}
Thus $V$ is finite almost everywhere, and we will show that
$
e^{i\frac{1}{k}\Box}V
$
diverges on $[\tfrac{1}{16},\tfrac{1}{4}]\times[-\tfrac{1}{16},\tfrac{1}{16}].$

This potential is an adaptation of an initial datum for the time-dependent nonelliptic Schr\"odinger equation considered in \cite{RVV}. The initial datum there was not real, the diverging sequence of time was allowed to depend on the point $x$, and more crucially, the initial datum was not compactly supported. Thus our arguments will have a different flavour, working on the frequency and spatial side simultaneously.

By changes of variables and the Fourier inversion formula,
\begin{align*}
e^{it\Box}[V_j^+](x)&= \frac{2^{(1-\beta) j}e^{i 2^jx_2}}{(2\pi)^2}\int \widehat{\phi}(\xi_1)\widehat{\phi}(\xi_2)\, e^{-i2^{j+1}t\xi_1\xi_2} e^{i(2^{j}\xi_1(x_1-2^{j+1}t)+\xi_2x_2)} d\xi\\
&=\frac{2^{(1-\beta) j}e^{i 2^jx_2}}{2\pi}\int
\phi\big(2^j(x_1-2^{j+1}t -2t\xi_2)\big)\widehat{\phi}(\xi_2)\,
e^{i\xi_2x_2} d\xi_2.
\end{align*}
Taking $t=1/k$ with $k$ the nearest natural number to $2^{j+1}/x_1$,
\begin{align*}
e^{i\frac{1}{k}\Box}[V_j^+](x) &= \frac{2^{(1-\beta) j}e^{i
2^jx_2}}{2\pi}\int
\phi\big(\zeta(x_1,j)-\tfrac{2^{j+1}}{k}\xi_2\big)\widehat{\phi}(\xi_2)\,
e^{i\xi_2x_2} d\xi_2,
\end{align*}
where $|\zeta(x_1,j)|\le \tfrac{1}{4}$ when  $x_1\in
[\tfrac{1}{16},\tfrac{1}{4}]$, so that, using the compact support
of~$\phi$, we see that
\begin{align*}
|e^{i\frac{1}{k}\Box}[V_j^+](x)|&=
\Big|\frac{2^{(1-\beta) j}}{2\pi}\int_{-16}^{16} \phi\big(\zeta(x_1,j)-\tfrac{2^{j+1}}{k}\xi_2\big)\widehat{\phi}(\xi_2)\, e^{i\xi_2x_2} d\xi_2\Big|
\\
&\ge \frac{2^{(1-\beta) j}}{2\pi} \Big |\int_{-16}^{16} \phi\big(\zeta(x_1,j)-\tfrac{2^{j+1}}{k}\xi_2\big)\widehat{\phi}(\xi_2)\, \cos(\xi_2x_2)\, d\xi_2\Big|.
\end{align*}
Now when $x_2\in [-\tfrac{1}{16},\tfrac{1}{16}]$, we have
$|\xi_2x_2|\le 1$, so that $|\cos (\xi_2 x_2)|>\cos(1)$. Using the
fact that $\phi$ and $\widehat{\phi}$ are nonnegative, we obtain
\begin{align*}
|e^{i\frac{1}{k}\Box}[V_j^+](x)| & \ge \frac{2^{(1-\beta) j}
\cos(1)}{2\pi}  \int_{-16}^{16}
\phi\big(\zeta(x_1,j)-\tfrac{2^{j+1}}{k}\xi_2\big)\widehat{\phi}(\xi_2)\,
d\xi_2
\\
&\ge C_12^{(1-\beta) j}.
\end{align*}

It remains to bound from above  the solution associated to the other
pieces of the potential. Again, by the Fourier inversion formula,
\begin{align*}
|e^{i\frac{1}{k}\Box}[V^\pm_{\ell}](x)|&= \frac{2^{(1-\beta) {\ell}}}{(2\pi)^2}\Big|\int \widehat{\phi}(\xi_1)\widehat{\phi}(\xi_2)\, e^{-i\frac{2}{k}\xi_1\xi_2} e^{i(2^{{\ell}}\xi_1(x_1 \mp \frac{2^{{\ell+1}}}{k})+\xi_2x_2)} d\xi\Big|\\
&= \frac{2^{(1-\beta) {\ell}}}{2\pi}\Big|\int \phi\big(2^{\ell}( x_1\mp \tfrac{2^{\ell+1}}{k}-\tfrac{2}{k}\xi_2)\big)\widehat{\phi}(\xi_2)\, e^{i\xi_2x_2} d\xi_2 \Big|.
\end{align*}
Using the fact that $\phi(y)\le C|y|^{-1/2}$, we obtain
$$
|e^{i\frac{1}{k}\Box}[V_{\ell}^\pm ](x)|\le C2^{(1/2-\beta) {\ell}}\int \frac{|\widehat{\phi}(\xi_2)|}{|x_1 \mp \tfrac{2^{\ell+1}}{k}-\tfrac{2}{k}\xi_2|^{1/2}}\,d\xi_2.
$$
Taking $0<\epsilon<\min\{1/4,1-\beta\}$, and using the rapid decay of $\widehat{\phi}$, we see that
\begin{align*}
|e^{i\frac{1}{k}\Box}[V_{\ell}^\pm](x)|\le C2^{(1/2-\beta) {\ell}}\Big(\int_{|\xi_2|< 2^{\epsilon j}} \frac{1}{|x_1 \mp \tfrac{2^{\ell+1}}{k}-\tfrac{2}{k}\xi_2|^{1/2}}\,d\xi_2 + C2^{-j}\Big).
\end{align*}
Now one can check that when $\ell\neq j$ or $j=\ell$ and $\mp$ is an addition,
$$
|\tfrac{2}{k}\xi_2|\le\tfrac{3}{4}|x_1\mp\tfrac{2^{\ell+1}}{k}|
$$
when $|\xi_2|\le 2^{j\epsilon}$. Indeed, when $j>\ell$, the
left-hand side is less than $\frac{1}{4}|x_1|$ which is less than
the right-hand side. On the other hand, when $j<\ell$ or $j=\ell$
and $\mp$ is an addition, the left-hand side is less than
$\frac{1}{2}|x_1|$ which is less than the right-hand side. Thus,
the integrand of the final integral is nonsingular so that the
integral is bounded by $C|x_1|^{-1/2}2^{\epsilon j}\le C2^{\epsilon
j}$.

By summing a geometric series in ${\ell}$, we obtain
\begin{align*}
\Big|\sum_{{\ell}\neq j} e^{i\frac{1}{k}\Box}V^\pm_{\ell}(x)+  e^{i\frac{1}{k}\Box}V^-_{j}(x)\Big|\le C_2 {2^{\epsilon j}} ,
\end{align*}
and we can conclude that on $[\tfrac{1}{16},\tfrac{1}{4}]\times[-\tfrac{1}{16},\tfrac{1}{16}]$,
$$
|e^{i\frac{1}{k}\Box}[V]|\ge |e^{i\frac{1}{k}\Box}[V_j^+]| - \Big|\sum_{{\ell}\neq j} e^{i\frac{1}{k}\Box}[V^\pm _{\ell}] +  e^{i\frac{1}{k}\Box}V^-_{j}(x)\Big|
\ge C_12^{j(1-\beta)}-C_22^{j\epsilon},
$$
which diverges as $j$ tends to infinity. Considering forty-five degree rotations of the $V_j$, which are Schwartz functions, via the pointwise equality, this yields
$$
|\T_1^{k}[V]|\ge |\T_1^{k}[V_j^+]| - \Big|\sum_{{\ell}\neq j}\T_1^{k}[V^\pm_{\ell}] +\T_1^{k}[V^-_{j}] \Big| \ge C_12^{j(1-\beta)}-C_22^{j\epsilon}
$$
on a forty-five degree rotation of $[\tfrac{1}{16},\tfrac{1}{4}]\times[-\tfrac{1}{16},\tfrac{1}{16}]$, so that $|\T_1^{k}[V]|$ diverges as $k$ tends to infinity.
Thus, by Theorem~\ref{Remainder}, combined with Alessandrini's identity \eqref{ale},
\begin{align*}\Big\{ x \,:\, \tfrac{k}{4\pi} \Big\langle (\Lambda_V-\Lambda_0)[u_{{k,x}}|_{\partial\D}], e^{i {\overline{\psi}}}|_{\partial\D} \Big\rangle\not\to V(x)\ \ \text{as}\ \ k\to\infty\Big\}
\end{align*}
contains a forty-five degree rotation of $[\tfrac{1}{16},\tfrac{1}{4}]\times[-\tfrac{1}{16},\tfrac{1}{16}]$, which has nonzero Lebesgue measure.\hfill$\Box$

\vspace{1em}

Note that this result is stable in the sense that $k\in\mathbb{N}$ can be replaced by any sequence $\{n_k\}_{k\in\mathbb{N}}$ satisfying $n_{k}\in[k,k+1)$.

\begin{rem} In  \cite{Sj}, Sj\"olin asked for which values of $s$ is it true that $$\lim_{k\to\infty} e^{i\frac{1}{k}\Delta} f(x)=0, \quad \text{a.e.}\ x\in\R^d \backslash (\text{supp} f),$$for all $f\in H^s$. In principle, this question could have stronger positive results and weaker negative results than Carleson's question: for which values of $s$ is it true that $$\lim_{k\to\infty} e^{i\frac{1}{k}\Delta} f(x)=f(x), \quad \text{a.e.}\ x\in\R^d,$$for all $f\in H^s$?
Indeed, before Bourgain's recent breakthrough \cite{Bo}, Sj\"olin
proved a stronger positive result for his question than what was
known for Carleson's question in three dimensions. Here we solve
Sj\"olin's question completely for the nonelliptic equation in two
dimensions. That is to say,
$$\lim_{k\to \infty} e^{i\frac{1}{k}\Box} f(x)=0, \quad \text{a.e.}\ x\in\R^2 \backslash (\text{supp} f),$$for all $f\in H^s$ if and only if $s\ge 1/2$.
\end{rem}

\appendix

\section{The {\small DN} map from the scattering amplitude}\label{samp}

It is well--known that in the absence of zero Dirichlet eigenvalues there is a unique weak solution to the Dirichlet problem \eqref{dp} that satisfies
 \begin{equation}\label{ref}
 \|u\|_{H^1(\Omega)}\le C\|f\|_{H^{1/2}(\partial \Omega)}
 \end{equation}
(see for example \cite{DKS} - in two dimensions $L^{n/2}(\R^n)$ can be replaced by~$L^2(\R^2)$).
Here $H^{1/2}(\partial \Omega):=H^1(\Omega)/H^1_0(\Omega)$, where $H^1_0(\Omega)$ denotes the closure of
$C^\infty_0(\Omega)$ in $H^1(\Omega)$.
The {\small DN} map $\Lambda_V$ is then defined by
\[ \Big \langle \Lambda_V[f],\psi \Big \rangle = \int_{\partial \Omega} \Lambda_V[f]\,\psi=\int_{\Omega} V u\Psi +\nabla u\cdot \nabla \Psi,\]
for all $\Psi\in H^1(\Omega)$ with $\psi=\Psi+H^1_0(\Omega)$.
When the solution and boundary are sufficiently regular, this
definition coincides with that of the introduction by Green's
formula.
To see that $\Lambda_V$ maps from $H^{1/2}(\partial \Omega)$ to  $H^{-1/2}(\partial \Omega)$, the dual of  $H^{1/2}(\partial \Omega)$, we note that by H\"older's inequality and the Hardy--Littlewood--Sobolev inequality,
\begin{align*}
\Big|\Big \langle \Lambda_V[f],\psi \Big \rangle\Big|&\le \|u\|_{H^1(\Omega)}\|\Psi\|_{H^1(\Omega)}+\|V\|_2\|u\|_{L^4(\Omega)}\|\Psi\|_{L^4(\Omega)}\\
&\le (1+C\|V\|_2)\|u\|_{H^1(\Omega)}\|\Psi\|_{H^1(\Omega)}
\end{align*}
whenever $\Psi\in H^1(\Omega)$, so that by \eqref{ref}, we obtain
\begin{align*}
\Big|\Big \langle \Lambda_V[f],\psi \Big \rangle\Big|&\le C (1+\|V\|_2)\|f\|_{H^{1/2}(\partial\Omega)}\|\psi\|_{H^{1/2}(\partial\Omega)}.
\end{align*}

There are a number of different approaches to showing that the scattering amplitude at a fixed energy $\kappa^2>0$ uniquely determines the {\small DN} map $\Lambda_{V-\kappa^2}$ and {\it vice versa} (see for example \cite{Be, Nachman88, U, SU93a, S}).
Here we follow a constructive argument due to  Nachman \cite[Section 3]{Nachman91}. We must additionally assume that $\kappa^2$ is not a Dirichlet eigenvalue of $-\Delta+V$. This can be  arranged by taking $\Omega$ sufficiently large as the eigenvalues decrease strictly as the domain grows \cite{NPC} (the result of \cite{L} can be extended to $L^2$-potentials using the unique continuation of \cite{JK}). We also additionally suppose that $V$ is real.

Let  $G_{V}$ and $G_{0}$ be the outgoing Green's functions that satisfy
\[(-\Delta+V-\kappa^2)G_{V}(x,y)=\delta(x-y),\quad  (-\Delta-\kappa^2) G_{0}(x,y)=\delta(x-y), \]
and let  $S_{V}$ and $S_{0}$ be the corresponding near-field operators defined via single layer potentials;
\[S_{V}[f](x)=\int_{\partial \Omega}  G_{V}(x,y) f(y) \,dy, \quad S_{0}[f](x)=\int_{\partial \Omega}  G_0(x,y) f(y) \,dy.\]
These are bounded and invertible,  mapping $H^{-1/2}(\partial\Omega)$ to $H^{1/2}(\partial\Omega)$ (the two--dimensional  proof can be found in \cite[Proposition A.1]{IN}). Then Nachman's formula \cite{Nachman88},
\[\Lambda_{V-\kappa^2}=\Lambda_{-\kappa^2}+S_{V}^{-1}-S_{0}^{-1}, \]
allows us to recover the {\small DN} map  on Lipschitz domains.

Thus it remains to recover the single layer potential $S_V$ from the scattering amplitude $A_V$ at energy $\kappa^2$.
For $\omega\in \mathbb{S}^1$, the outgoing scattering solution $v(\cdot,\omega,\kappa)$ is the unique solution to the Lippmann--Schwinger equation
\begin{equation}\label{ls0}
v(y,\omega,\kappa)=e^{i\kappa y \cdot \omega}-\int_{\R^2} G_0(y,z)V(z) v(z,\omega,\kappa )\, dz.\end{equation}
For $(\sigma,\omega)\in \mathbb{S}^1\times \mathbb{S}^1$, the scattering amplitude then satisfies
\begin{equation}\label{f}
A_V(\sigma,\omega,\kappa)=\int_{\R^2} e^{-i\kappa \sigma\cdot z} V(z) v(z,\omega,\kappa)\, dz.
\end{equation}

 When $\Omega$ is a disc, Nachman recovers $S_V$ via formulae given by expansions in spherical harmonics as below. Otherwise he uses a density argument (we remark that Sylvester \cite{S} also invokes density in order to recover). Since we have been obliged to work with $\Omega$ a square, at this point we deviate and instead follow an argument of Stefanov \cite{Stefanov90}, obtaining an explicit formula for the Green's function $G_V$ in terms of $A_V$. Alternatively it seems likely that we could pass to the {\small DN} map on the square from that on the disc via the argument in \cite[Section 6]{Nachman96} for the conductivity problem, but we prefer this more direct approach.

 Stefanov worked in three dimensions, with bounded potentials, and a number of details change in two dimensions, so we present the argument.
 We recover $G_V$ outside of a disc which contains the potential, but which is contained in the domain, so that $S_V$ can be obtained by integrating along the sides of our square~$\Omega$.
 
First we require the following asymptotics. 
\begin{lem}\label{A}
\[G_V(x,y)-G_0(x,y)=\frac{-i}{8\pi\kappa}\frac{e^{i \kappa |x|}}{|x|^\frac12}  \frac{e^{i \kappa |y|}}{|y|^\frac12} A_V\Big(-\frac{x}{|x|},\frac{y}{|y|},\kappa\Big) +o\Big(\frac{1}{|x|^\frac12|y|^\frac12}\Big).\]
\end{lem}
\begin{proof} It is well--known (see for example (3.66) in \cite{NPT}) that  $G_V$ satisfies
\begin{equation}\label{GV}  G_V(x,z)=  \frac{e^{i\frac{\pi}{4}}}{(8\pi)^{\frac{1}{2}}}  \frac {e^{i \kappa  |x|}}{ \kappa ^\frac{1}{2}|x|^\frac12}  v\Big(z,-\frac{x}{|x|},\kappa\Big) +o\Big(\frac{1}{|x|^{\frac{1}{2}}}\Big),
 \end{equation}
and, in particular,
\begin{equation}\label{kd}
G_0(y,z)=\frac{e^{i\frac{\pi}{4}}}{(8\pi)^{\frac{1}{2}}}  \frac {e^{i \kappa |y|}}{\kappa^{\frac{1}{2}} |y|^\frac12} e^{- i \kappa \frac{y}{|y|} \cdot z } +o\Big(\frac{1}{|y|^{\frac{1}{2}}}\Big).\end{equation}
On the other hand, it is easy to verify that \begin{equation}\label{lsr}G_V(x,y)-G_0(x,y)=-\int_{\R^2} G_V(x,z) V(z) G_0(y,z)\, dz. \end{equation}
Substituting in \eqref{GV} and \eqref{kd},  see that $G_V(x,y)-G_0(x,y)$ is equal to
\[\frac{-i}{8\pi\kappa}   \frac{e^{i \kappa  |x|}}{ |x|^\frac12}  \frac{e^{i \kappa  |y|}}{|y|^\frac12}  \int e^{-i\kappa\frac{y}{|y|}\cdot z} V(z) v\Big(z,-\frac{x}{|x|},\kappa \Big) dz +o\Big(\frac{1}{|x|^\frac12|y|^\frac12}\Big),  \]
so that by \eqref{f} we obtain the result.
\end{proof}

In the following, $J_n$ and $H_n^{(1)}$ denote the Bessel and Hankel functions  of the first kind of $n$th order, respectively (see for example \cite{leb}). We also write $x$ in polar coordinates as $(|x|,\phi_x)$.

\begin{thm} Let $V\in H^s$ with $s>0$ be supported in the disc of radius~$\rho$, centred at the origin, and consider the Fourier series
\[
A_V(\sigma,\omega,\kappa)= \sum_{n\in\mathbb{Z}} \sum _{m\in \mathbb{Z}}  a_{n,m} e^{in \phi_\sigma} e^{i m \phi_\omega}. \]
Then \begin{equation*}\label{explicit}  G_V(x,y)-G_0(x,y)= \!\!\sum_{n\in\mathbb{Z}}\sum_{m\in\mathbb{Z}} \!\!\frac{(-1)^{n} }{16}i^{n+m}a_{n,m} H_n^{(1)}(\kappa |x|) H_m^{(1)} (\kappa |y|) e^{in \phi_x} e^{i m \phi_y},
\end{equation*}
where the series is uniformly, absolutely convergent for $|x|>|y|>R>\frac{3}{2}\rho$.
\end{thm}

\begin{proof}
We can expand $G_0(x,y)=\frac{i}{4}H_0^{(1)}(\kappa|x-y|)$ as
\[G_0(x,y)= \frac{i}{4}\Big(H_0^{(1)}(\kappa |x|) J_0(\kappa|y|)+
2\sum_{n\ge1} H_n^{(1)}(\kappa |x|) J_n(\kappa|y|)\cos(\phi_x-\phi_y)\Big), \]
(see for example \cite[Section 3.4]{ColtonKressbook92} or \cite[Theorem 3.4]{ruiz}). As $H_{-n}^{(1)}=(-1)^nH_n^{(1)}$ and $J_{-n}=(-1)^{n} J_n$, in order to separate variables it will be convenient to write this as
$$
G_0(x,y)=\frac{i}{4}\sum_{n\in\mathbb{Z}} H_n^{(1)}(\kappa |x|) J_n(\kappa|y|)e^{in \phi_x} e^{-i n\phi_y}.
$$
As before, it is easy to check that
\begin{equation*}\label{ls}G_V(x,y)-G_0(x,y)=-\int_{\R^2} G_0(x,z) V(z) G_V(z,y)\, dz, \end{equation*}
and so substituting  \eqref{lsr} into this we obtain $G_V-G_0=-I_1+I_2$, where
\begin{align*}
I_1&=\int G_0(x,z) V(z)G_0(z,y)\, dz\\
I_2&=\int G_0(x,z_1)V(z_1) \int G_V(z_1,z_2) V(z_2) G_0(y,z_2)\, dz_2dz_1.
\end{align*}
Now in both integrals we introduce the expansion of $G_0$ (note that $G_0(x,y)=G_0(y,x)$), extracting the terms independent of $z,z_1,z_2$. In this way we get
\begin{align}
I_1&=-\frac{1}{16}\sum_{n\in\mathbb{Z}}  \sum_{m\in\mathbb{Z}} \alpha_{n,m}H_n^{(1)}(\kappa |x|)  H^{(1)}_m(\kappa|y|)e^{in \phi_x} e^{i m \phi_y},\label{non0}\\
I_2&= -\frac{1}{16}\sum_{n\in \mathbb{Z}}  \sum_{m\in \mathbb{Z}} \beta_{n,m}H_n^{(1)}(\kappa |x|)  H^{(1)}_m(\kappa|y|) e^{in \phi_x} e^{i m \phi_y},\label{non}
\end{align}
where
\begin{align*}
\alpha_{n,m}&= \int_{\R^2} V(z) J_n(\kappa|z|) J_m(\kappa|z|)e^{-i(n+m) \phi_z}\,dz,\\
\beta_{n,m}&=\int_{\R^4} J_n(\kappa|z_1|)  V(z_1)G_V(z_1,z_2) V(z_2)  J_m(\kappa|z_2|) e^{-in\phi_{z_1}}e^{-im\phi_{z_2}}\,dz_1dz_2 .
\end{align*}
It remains to show that the sums \eqref{non0} and \eqref{non} converge uniformly and absolutely for $|x|>|y|>R>\frac{3}{2}\rho$. Once we know that this  is the case, we can take limits and use the asymptotics of the Hankel functions for
large~$r$;
\[H_n^{(1)}(r) = e^{-i(n\frac{\pi}{2}+\frac{\pi}{4})}\Big(\frac{2}{\pi r}\Big)^\frac12    e^{ir} + o\Big(\frac{1}{r^\frac{1}{2}}\Big) \]
(see for example \cite[Section 5.16]{leb}),
and then Lemma~\ref{A} tells us that $$ -\frac{1}{16}(-i)^{n+m+1}\frac{2}{\pi}(\beta_{n,m}-\alpha_{n,m})=-i\frac{(-1)^n}{8\pi}a_{n,m}.$$

To see that the sums converge note that, by H\"older's inequality, we have
\begin{align*} |\alpha_{n,m}|&\le C_\rho\|V\|_{L^2}\|J_n(\kappa|\cdot|)\|_{L^\infty(B_\rho)} \|J_m(\kappa|\cdot|)\|_{L^\infty(B_\rho)},\\
|\beta_{n,m}|  &\le \|G_V\|_{L^2(B_\rho\times B_\rho)}   \|V\|^2_{L^2}\|J_n(\kappa|\cdot|)\|_{L^\infty(B_\rho)} \|J_m(\kappa|\cdot|)\|_{L^\infty(B_\rho)}.\end{align*}
At this point we deviate from \cite{Stefanov90} as there seems to be less local knowledge regarding $G_V$ in two dimensions. Instead we can rewrite \eqref{lsr} as
$$
G_V(\cdot,y)=G_0(\cdot,y)- (-\Delta+V-\kappa^2-i0)^{-1}[VG_0(\cdot,y)],
$$
and use that the resolvent is bounded from $L^2((1+|\cdot|^2)^{\delta})$ to $L^2((1+|\cdot|^2)^{-\delta})$ with $\delta>1/2$ (see \cite[Theorem 4.2]{agmon}). Thus, using that $V$ is compactly supported, and taking $\frac{1}{2}=\frac{1}{p}+\frac{1}{q}$ with large $p$ so that $1-\frac{2}{q}=s$,
\begin{align*}
\|G_V(\cdot,y)\|_{L^2(B_\rho)}&\le \|G_0(\cdot,y)\|_{L^2(B_\rho)}+ C_\rho\|VG_0(\cdot,y)\|_{L^2(B_\rho)}\\
&\le \|G_0(\cdot,y)\|_{L^2(B_\rho)}+ C_\rho\|V\|_q\|G_0(\cdot,y)\|_{L^p(B_\rho)}\\
&\le \|G_0(\cdot,y)\|_{L^2(B_\rho)}+ C_\rho\|V\|_{H^{s}}\|G_0(\cdot,y)\|_{L^p(B_\rho)},
\end{align*}
by the Hardy--Littlewood--Sobolev inequality. Integrating again with respect to $y$, and recalling that the singularity of $H_0^{(1)}$ at the origin is  logarithmic, we see that $\|G_V\|_{L^2(B_\rho\times B_\rho)}\le C$.
Then, using the Taylor series expansion for the Bessel function,
$$
|J_{n}(r)|=\Big|\sum_{j\ge 0} \frac{(-1)^j}{j!(|n|+j)!}\Big(\frac{r}{2}\Big)^{2j+|n|}\Big|\le C_{\rho}\frac{1}{|n|!}\Big(\frac{\rho}{2}\Big)^{|n|},\quad 0\le r\le \rho,$$
we see that
\begin{align*} |\alpha_{n,m}|&\le C_{\rho}\|V\|_{L^2}\frac{1}{|n|!}\Big(\frac{\rho}{2}\Big)^{|n|}\frac{1}{|m|!}\Big(\frac{\rho}{2}\Big)^{|m|},\\
|\beta_{n,m}|  &\le C_{\rho}(1+\|V\|^3_{H^s})\frac{1}{|n|!}\Big(\frac{\rho}{2}\Big)^{|n|}\frac{1}{|m|!}\Big(\frac{\rho}{2}\Big)^{|m|}.\end{align*}
Finally, we require the Hankel function estimate,
$$
|H^{(1)}_{n}(r)|\le C_R|n|!\Big(\frac{3}{R}\Big)^{|n|},\quad R\le r,
$$
which is proven in \cite[Lemma 2.3]{AFR}.
The sums \eqref{non0} and \eqref{non} are then bounded by  a constant multiple of
\[\sum_{n\ge 0}\sum_{m\ge 0} \Big(\frac{3\rho}{2R}\Big)^{n}\Big(\frac{3\rho}{2R}\Big)^{m} \]
which is convergent when $|x|>|y| >R>\frac{3}{2}\rho$, and so we are done.
\end{proof}

\vspace{1em}

The authors thank Juan Antonio Barcel\'o, Victor Arnaiz, Antonio C\'ordoba, Adrian Nachman, Gen Nakamura, Alberto Ruiz, Chris Ruiz, Mikko Salo, Plamen Stefanov and Jorge Tejero for helpful (electronic) conversations.

\end{document}